\documentclass{amsart}
\usepackage{amssymb}
\usepackage{amsfonts}

\newtheorem{theorem}{Theorem}
\theoremstyle{plain}

\newtheorem{definition}{Definition}
\newtheorem{example}{Example}

\newtheorem{lemma}{Lemma}

\newtheorem{proposition}{Proposition}
\newtheorem{remark}{Remark}

\numberwithin{equation}{section}
\numberwithin{theorem}{section}
\numberwithin{algorithm}{section}
\numberwithin{axiom}{section}
\numberwithin{case}{section}
\numberwithin{claim}{section}
\numberwithin{conclusion}{section}
\numberwithin{condition}{section}
\numberwithin{conjecture}{section}
\numberwithin{corollary}{section}
\numberwithin{criterion}{section}
\numberwithin{definition}{section}
\numberwithin{example}{section}
\numberwithin{exercise}{section}
\numberwithin{lemma}{section}
\numberwithin{notation}{section}
\numberwithin{problem}{section}
\numberwithin{proposition}{section}
\numberwithin{remark}{section}
\numberwithin{solution}{section}

\input{tcilatex}

\begin{document}
\title{The sphere covering inequality and its dual}
\author{Changfeng Gui}
\address{Department of Mathematics, The University of Texas at San Antonio,
San Antonio, TX 78249}
\email{changfeng.gui@utsa.edu}
\author{Fengbo Hang}
\address{Courant Institute, New York University, 251 Mercer Street, New York
NY 10012}
\email{fengbo@cims.nyu.edu}
\author{Amir Moradifam}
\address{Department of Mathematics, University of California, Riverside, CA
92521}
\email{moradifam@math.ucr.edu}
\date{}

\begin{abstract}

We present a new proof of the sphere covering inequality in the spirit of
comparison geometry, and as a byproduct we find another sphere covering
inequality which can be viewed as the dual of the original one. We also
prove sphere covering inequalities on surfaces satisfying general isoperimetric inequalities, and discuss their applications to elliptic equations with exponential nonlinearities in dimension two.  The approach in this paper extends, improves, and unifies several
inequalities about solutions of elliptic equations with
exponential nonlinearities.

\end{abstract}

\maketitle

\section{Introduction\label{sec1}}

Second order nonlinear elliptic equations with exponential nonlinearity of the form 
\begin{equation}
\Delta u+e^{u}=f(x) \ \  \hbox{in}\ \ \Omega \subset \mathbb{R}^2,
\end{equation}
arise in many important problems in mathematics, mathematical physics and biology.  Such equations have been extensively studied in the context of Moser-Trudinger inequalities, Chern-Simons self-dual vortices, Toda systems, conformal geometry, statistical mechanics of two-dimensional turbulence, self-gravitating cosmic strings, theory of elliptic functions and hyperelliptic curves and  free boundary models of cell motility, see  \cite{BFR, BL, BL2, BLT, BCLT, BT, BT2, Be, CLS, CaY, CLMP, CLMP2, CK, CFL, CY, CY2, CY3, DJLW, GL, L, L2, LM, LM2, LW, LW2, LWY, Y} and the references cited therein. 

The sphere covering inequality was recently introduced in \cite{GM}, and has been applied to solve various problems about symmetry and uniqueness of solutions of elliptic equations with exponential nonlinearity in dimension $n=2$. It particular, it was applied to prove  a long-standing conjecture of Chang-Yang (\cite{CY}) concerning the best constant in Moser-Trudinger type inequalities \cite{GM}, and has led to several symmetry and uniqueness results for mean field equations, Onsager vortices, Sinh-Gordon equation, cosmic string equation, Toda systems and rigidity of Hawking mass in general relativity \cite{GJM, GM, GM2, GM3, LLTY, SSTW, WWY}.

\begin{theorem}[The Sphere Covering Inequality {\protect\cite[Theorem 1.1]{GM}}]
\label{thm1.1}Let $\Omega _{0}\subset \mathbb{R}^{2}$ be a simply connected
domain. Assume $u_{1}\in C^{2}\left( \overline{\Omega _{0}}\right) $ such
that%
\begin{equation}
\Delta u_{1}+e^{2u_{1}}\geq 0\text{ on }\Omega _{0},\quad \int_{\Omega
_{0}}e^{2u_{1}}dx\leq 4\pi .
\end{equation}%
Let $\Omega \subset \Omega _{0}$ be a bounded open set. Assume $u_{2}\in
C^{2}\left( \overline{\Omega }\right) $ such that%
\begin{eqnarray*}
\Delta u_{2}+e^{2u_{2}} &\geq &\Delta u_{1}+e^{2u_{1}}\text{ in }\Omega 
\text{,} \\
u_{2} &>&u_{1}\text{ in }\Omega , \\
\left. u_{2}\right\vert _{\partial \Omega } &=&\left. u_{1}\right\vert
_{\partial \Omega }.
\end{eqnarray*}%
Then%
\begin{equation}
\int_{\Omega }e^{2u_{1}}dx+\int_{\Omega }e^{2u_{2}}dx\geq 4\pi .
\end{equation}
\end{theorem}

In this paper, we present an approach that completes, simplifies and improves the sphere covering inequality and several other inequalities about solutions of the elliptic equations with exponential nonlinearities. In particular,  we will prove the following generalization of the sphere covering inequality with a method different from the one in \cite{GM}. 

\begin{theorem}
\label{thm1.1variant}Let $\Omega _{0}\subset \mathbb{R}^{2}$ be a simply
connected domain. Assume $u_{1}\in C^{2}\left( \overline{\Omega _{0}}\right) 
$ such that%
\begin{equation}
\Delta u_{1}+e^{2u_{1}}\geq 0\text{ on }\Omega _{0},\quad \int_{\Omega
_{0}}e^{2u_{1}}dx\leq 4\pi .
\end{equation}%
Let $\Omega \subset \Omega _{0}$ be a bounded open set. Assume $u_{2}\in
C^{2}\left( \overline{\Omega }\right) $ and $0<\lambda \leq 1$ such that%
\begin{eqnarray*}
\Delta u_{2}+\lambda e^{2u_{2}} &\geq &\Delta u_{1}+e^{2u_{1}}\text{ in }%
\Omega \text{,} \\
u_{2} &>&u_{1}\text{ in }\Omega , \\
\left. u_{2}\right\vert _{\partial \Omega } &=&\left. u_{1}\right\vert
_{\partial \Omega }.
\end{eqnarray*}%
Then%
\begin{equation}
\int_{\Omega }e^{2u_{1}}dx+\int_{\Omega }e^{2u_{2}}dx\geq \frac{4\pi }{%
\lambda }.
\end{equation}
\end{theorem}

We shall also prove the following inequality which can be viewed as the dual of the sphere covering
inequality.

\begin{theorem}
\label{thm1.2}Let $\Omega _{0}\subset \mathbb{R}^{2}$ be a simply connected
domain. Assume $u_{1}\in C^{2}\left( \overline{\Omega _{0}}\right) $ such
that%
\begin{equation}
\Delta u_{1}+e^{2u_{1}}\geq 0\text{ on }\Omega _{0},\quad \int_{\Omega
_{0}}e^{2u_{1}}dx\leq 4\pi .
\end{equation}%
Let $\Omega \subset \Omega _{0}$ be a bounded open set. Assume $u_{2}\in
C^{2}\left( \overline{\Omega }\right) $ such that%
\begin{eqnarray*}
\Delta u_{2}+e^{2u_{2}} &\leq &\Delta u_{1}+e^{2u_{1}}\text{ in }\Omega 
\text{,} \\
u_{2} &<&u_{1}\text{ in }\Omega , \\
\left. u_{2}\right\vert _{\partial \Omega } &=&\left. u_{1}\right\vert
_{\partial \Omega }.
\end{eqnarray*}%
Then%
\begin{equation}
\int_{\Omega }e^{2u_{1}}dx+\int_{\Omega }e^{2u_{2}}dx\geq 4\pi .
\end{equation}
\end{theorem}

We will develop an approach for Theorem \ref{thm1.1variant} which is different from
the one in \cite{GM}, and shall modify it to prove Theorem \ref{thm1.2}.  Our
method has the general spirit of comparison geometry. Under the
assumption of Theorem \ref{thm1.1variant}, let $g=e^{2u_{1}}\left\vert
dx\right\vert ^{2}$, here $\left\vert dx\right\vert ^{2}$ is the Euclidean
metric. Then the Gauss curvature $K\leq 1$ and the area $\mu \left( \Omega
_{0}\right) \leq 4\pi $,  where $\mu (E) $ is the measure of $E$ associated with the metric $g$. 
It follows from \cite{B} and \cite[Lemma 4.2]{CCL}
that the following isoperimetric inequality holds on $\left( \Omega _{0},g\right) $ 
for any domain $E$ in $\Omega _{0}$ with smooth boundary 
\begin{equation}
4\pi \mu \left( E\right) -\mu^2\left( E\right)\leq s^{2} \left( \partial
E\right),
\end{equation}%
where $s$ is the $1$-dimensional measure associated with $g$. Using $g$ as a
background metric, we can rewrite the differential inequality between $u_{1}$
and $u_{2}$ into a differential inequality involving $u=u_{2}-u_{1}$.
Applying ideas from \cite{B,S} to the resulting differential inequality on $\left(
\Omega _{0},g\right) $ gives us an inequality which will imply Theorem \ref%
{thm1.1variant}. Indeed the proof of Theorem \ref{thm1.1variant} is based on the following more general result. 

\begin{theorem}
\label{thm1.3}Let $\left( M,g\right) $ be a simply connected smooth Riemann
surface. Assume $K\leq 1$ and $\mu \left( M\right) \leq 4\pi $, here $K$ is
the Gauss curvature and $\mu $ is the measure of $\left( M,g\right) $. Let $%
\Omega $ be a domain with compact closure and nonempty boundary, and $%
\lambda $ be a constant. If $u\in C^{2 }\left( \overline{\Omega }%
\right) $ such that $u>0$ in $\Omega $ and%
\begin{equation}
-\Delta _{g}u+1\leq \lambda e^{2u},\quad \left. u\right\vert _{\partial
\Omega }=0.
\end{equation}%
Then%
\begin{equation}
4\pi \int_{\Omega }e^{2u}d\mu -\lambda \left( \int_{\Omega }e^{2u}d\mu
\right) ^{2}\leq 4\pi \mu \left( \Omega \right) -\mu ^{2} \left( \Omega \right).
\end{equation}%
In particular if $0<\lambda \leq 1$, then%
\begin{equation}
\int_{\Omega }e^{2u}d\mu +\mu \left( \Omega \right) \geq \frac{4\pi }{%
\lambda }.  \label{eqmain}
\end{equation}
\end{theorem}

It is interesting that in this comparison theorem, what is compared is not
the area itself, but the quantity $4\pi \mu \left( \Omega \right) -\mu^{2}
\left( \Omega \right)$, which is exactly the quantity appeared in the
isoperimetric inequality.\\

The proof of Theorem \ref{thm1.2} also follows from the following more general result. 

\begin{theorem}
\label{thm1.4}Let $\left( M,g\right) $ be a simply connected smooth Riemann
surface. Assume $K\leq 1$ and $\mu \left( M\right) \leq 4\pi $, here $K$ is
the Gauss curvature and $\mu $ is the measure of $\left( M,g\right) $. Let $%
\Omega $ be a domain with compact closure and nonempty boundary, and $%
\lambda $ be a constant. If $u\in C^{2 }\left( \overline{\Omega }%
\right) $ such that $u<0$ in $\Omega $ and%
\begin{equation}
-\Delta _{g}u+1\geq \lambda e^{2u},\quad \left. u\right\vert _{\partial
\Omega }=0.
\end{equation}%
Then%
\begin{equation}
4\pi \int_{\Omega }e^{2u}d\mu -\lambda \left( \int_{\Omega }e^{2u}d\mu
\right) ^{2}\geq 4\pi \mu \left( \Omega \right) -\mu ^{2} \left( \Omega \right).
\end{equation}%
In particular if $\lambda =1$, then%
\begin{equation}
\int_{\Omega }e^{2u}d\mu +\mu \left( \Omega \right) \geq 4\pi .
\end{equation}
\end{theorem}

In Section 2, we present proofs of Theorems \ref{thm1.1variant}, \ref{thm1.2}, \ref{thm1.3}, and \ref{thm1.4}. In section 3, we will prove sphere covering inequalities on surfaces satisfying general isoperimetric inequalities and shall discuss their applications to elliptic equations with exponential nonlinearities. 

\section{Differential inequalities on surface with curvature at most $1$%
\label{sec2}}

In this section we will prove Theorem \ref{thm1.3} and \ref{thm1.4}. The
main point is  that the approach in \cite{B,S} can be performed on simply
connected surface with curvature at most $1$.

\begin{proof}[Proof of Theorem \protect\ref{thm1.3}]
By approximation and replacing $\lambda $ with $\lambda +\varepsilon $, $%
\varepsilon $ is a small positive number, we can assume $u$ is a Morse
function. For $t>0$, let%
\begin{equation*}
\alpha \left( t\right) =\int_{\{ u>t\}}e^{2u}d\mu ,\quad \beta \left( t\right)
=\int_{\{ u>t\}}d\mu .
\end{equation*}%
By co-area formula we get%
\begin{eqnarray*}
\alpha \left( t\right) &=&\int_{\{ u>t\}}e^{2u}\frac{\left\vert \nabla
u\right\vert }{\left\vert \nabla u\right\vert }d\mu \\
&=&\int_{t}^{\infty }d\tau \int_{\{ u=\tau \}}\frac{e^{2u}}{\left\vert \nabla
u\right\vert }ds  \notag \\
&=&\int_{t}^{\infty }\left( e^{2\tau }\int_{\{ u=\tau \}}\frac{ds}{\left\vert
\nabla u\right\vert }\right) d\tau  \notag
\end{eqnarray*}%
and%
\begin{equation*}
\beta \left( t\right) =\int_{\{ u>t\}}\frac{\left\vert \nabla u\right\vert }{%
\left\vert \nabla u\right\vert }d\mu =\int_{t}^{\infty }d\tau \int_{\{ u=\tau \}}%
\frac{ds}{\left\vert \nabla u\right\vert }.
\end{equation*}%
It follows that%
\begin{equation*}
\alpha ^{\prime }\left( t\right) =-e^{2t}\int_{\{ u=t\}}\frac{ds}{\left\vert
\nabla u\right\vert },
\end{equation*}%
and%
\begin{equation*}
\beta ^{\prime }\left( t\right) =-\int_{\{ u=t\}}\frac{ds}{\left\vert \nabla
u\right\vert }.
\end{equation*}%
In particular%
\begin{equation*}
\alpha ^{\prime }\left( t\right) =e^{2t}\beta ^{\prime }\left( t\right) .
\end{equation*}%
On the other hand, by the differential inequality we have%
\begin{equation*}
\int_{\{ u>t\}}\left( -\Delta u\right) d\mu +\beta \left( t\right) \leq \lambda
\alpha \left( t\right) ,
\end{equation*}%
hence%
\begin{equation*}
\int_{\{ u=t\}}\left\vert \nabla u\right\vert ds\leq \lambda \alpha -\beta .
\end{equation*}%
Multiplying both sides by $-\alpha'(t)$ we get
\begin{equation*}
e^{2t}\int_{\{ u=t\}}\frac{ds}{\left\vert \nabla u\right\vert }%
\int_{\{ u=t\}}\left\vert \nabla u\right\vert ds\leq -\lambda \alpha \alpha
^{\prime }+e^{2t}\beta \beta ^{\prime },
\end{equation*}%
which implies%
\begin{equation*}
e^{2t}s\left( \left\{ u=t\right\} \right) ^{2}\leq -\lambda \alpha \alpha
^{\prime }+e^{2t}\beta \beta ^{\prime }.
\end{equation*}%
Applying the isoperimetric inequality on $\left( M,g\right) $ (see \cite{B}
, \cite[Lemma 4.2]{CCL}) we get 
\begin{equation*}
e^{2t}\left( 4\pi \beta -\beta ^{2}\right) \leq -\lambda \alpha \alpha
^{\prime }+e^{2t}\beta \beta ^{\prime }.
\end{equation*}%
Consequently 
\begin{equation*}
4\pi \left( e^{2t}\right) ^{\prime }\beta -\left( e^{2t}\right) ^{\prime
}\beta ^{2}\leq -\lambda \left( \alpha ^{2}\right) ^{\prime }+e^{2t}\left(
\beta ^{2}\right) ^{\prime }.
\end{equation*}%
In other words 
\begin{equation*}
4\pi \left[ \left( e^{2t}\beta \right) ^{\prime }-\alpha ^{\prime }\right]
\leq -\lambda \left( \alpha ^{2}\right) ^{\prime }+\left( e^{2t}\beta
^{2}\right) ^{\prime },
\end{equation*}%
and hence%
\begin{equation*}
4\pi \left( \alpha -e^{2t}\beta \right) -\lambda \alpha ^{2}+e^{2t}\beta ^{2}%
\text{ is increasing.}
\end{equation*}%
Thus
\begin{equation*}
4\pi \left( \alpha \left( 0\right) -\beta \left( 0\right) \right) -\lambda
\alpha \left( 0\right) ^{2}+\beta \left( 0\right) ^{2}\leq 0.
\end{equation*}%
In other words 
\begin{equation*}
4\pi \int_{\Omega }e^{2u}d\mu -\lambda \left( \int_{\Omega }e^{2u}d\mu
\right) ^{2}\leq 4\pi \mu \left( \Omega \right) -\mu ^{2} \left( \Omega \right).
\end{equation*}
When $0<\lambda \leq 1$, we have%
\begin{equation*}
4\pi \int_{\Omega }e^{2u}d\mu -\lambda \left( \int_{\Omega }e^{2u}d\mu
\right) ^{2}\leq 4\pi \mu \left( \Omega \right) -\lambda \mu ^{2} \left( \Omega
\right),
\end{equation*}%
and%
\begin{eqnarray*}
&&4\pi \left( \int_{\Omega }e^{2u}d\mu -\mu \left( \Omega \right) \right) \\
&\leq &\lambda \left( \int_{\Omega }e^{2u}d\mu +\mu \left( \Omega \right)
\right) \left( \int_{\Omega }e^{2u}d\mu -\mu \left( \Omega \right) \right) .
\notag
\end{eqnarray*}%
Hence%
\begin{equation*}
\int_{\Omega }e^{2u}d\mu +\mu \left( \Omega \right) \geq \frac{4\pi }{%
\lambda }.
\end{equation*}
\end{proof}

We will derive Theorem \ref{thm1.4} by flipping all the inequalities.

\begin{proof}[Proof of Theorem \protect\ref{thm1.4}]
Again we can assume $u$ is a Morse function. For $t<0$, let%
\begin{equation*}
\alpha \left( t\right) =\int_{\{ u<t\}}e^{2u}d\mu ,\quad \beta \left( t\right)
=\int_{\{ u<t\}}d\mu .
\end{equation*}%
Then%
\begin{equation*}
\alpha ^{\prime }\left( t\right) =e^{2t}\int_{\{ u=t\}}\frac{ds}{\left\vert
\nabla u\right\vert },\quad \beta ^{\prime }\left( t\right) =\int_{\{ u=t\}}\frac{%
ds}{\left\vert \nabla u\right\vert },\quad \alpha ^{\prime }\left( t\right)
=e^{2t}\beta ^{\prime }\left( t\right) .
\end{equation*}%
On the other hand, since%
\begin{equation*}
\Delta u-1\leq -\lambda e^{2u},
\end{equation*}%
integrating on $\left\{ u<t\right\} $ we get%
\begin{equation*}
\int_{\{ u=t\}}\left\vert \nabla u\right\vert ds-\beta \left( t\right) \leq
-\lambda \alpha \left( t\right) .
\end{equation*}%
We have%
\begin{eqnarray*}
e^{2t}\left( 4\pi \beta -\beta ^{2}\right) &\leq &e^{2t}s^{2} \left( \left\{
u=t\right\} \right) \\
&\leq &e^{2t}\int_{\{ u=t\}}\frac{ds}{\left\vert \nabla u\right\vert }%
\int_{\{ u=t\}}\left\vert \nabla u\right\vert ds  \notag \\
&\leq &e^{2t}\beta \beta ^{\prime }-\lambda \alpha \alpha ^{\prime }.  \notag
\end{eqnarray*}%
Hence%
\begin{equation*}
e^{2t}\left( \beta ^{2}\right) ^{\prime }-\lambda \left( \alpha ^{2}\right)
^{\prime }\geq 4\pi \left( e^{2t}\right) ^{\prime }\beta -\left(
e^{2t}\right) ^{\prime }\beta ^{2}.
\end{equation*}%
It follows that%
\begin{equation*}
\left( e^{2t}\beta ^{2}\right) ^{\prime }-\lambda \left( \alpha ^{2}\right)
^{\prime }\geq 4\pi \left[ \left( e^{2t}\beta \right) ^{\prime }-\alpha
^{\prime }\right] ,
\end{equation*}%
and%
\begin{equation*}
4\pi \left( \alpha -e^{2t}\beta \right) +e^{2t}\beta ^{2}-\lambda \alpha ^{2}%
\text{ is increasing.}
\end{equation*}%
In particular%
\begin{equation*}
4\pi \left( \alpha \left( 0\right) -\beta \left( 0\right) \right) +\beta
\left( 0\right) ^{2}-\lambda \alpha \left( 0\right) ^{2}\geq 0.
\end{equation*}%
In other words 
\begin{equation*}
4\pi \int_{\Omega }e^{2u}d\mu -\lambda \left( \int_{\Omega }e^{2u}d\mu
\right) ^{2}\geq 4\pi \mu \left( \Omega \right) -\mu ^{2} \left( \Omega \right).
\end{equation*}%
When $\lambda =1$, we have%
\begin{eqnarray*}
&&4\pi \left( \mu \left( \Omega \right) -\int_{\Omega }e^{2u}d\mu \right) \\
&\leq &\left( \mu \left( \Omega \right) +\int_{\Omega }e^{2u}d\mu \right)
\left( \mu \left( \Omega \right) -\int_{\Omega }e^{2u}d\mu \right) .  \notag
\end{eqnarray*}%
Hence%
\begin{equation*}
\int_{\Omega }e^{2u}d\mu +\mu \left( \Omega \right) \geq 4\pi .
\end{equation*}
\end{proof}

Theorem \ref{thm1.1variant} easily  follows from Theorem \ref{thm1.3}. 

\begin{proof}[Proof of Theorem \protect\ref{thm1.1variant}]
Let $g=e^{2u_{1}}\left\vert dx\right\vert ^{2}$, then%
\begin{equation*}
K=-e^{-2u_{1}}\Delta u_{1}\leq 1
\end{equation*}%
and $\mu \left( \Omega _{0}\right) \leq 4\pi $. Let $u=u_{2}-u_{1}$. We have%
\begin{equation*}
-\Delta u\leq e^{2u_{1}}\left( \lambda e^{2u}-1\right) .
\end{equation*}%
Hence%
\begin{equation*}
-\Delta _{g}u\leq \lambda e^{2u}-1.
\end{equation*}%
Note that $u>0$ in $\Omega $ and $\left. u\right\vert _{\partial \Omega }=0$%
. Thus by Theorem \ref{thm1.3} we have
\begin{equation*}
\int_{\Omega }e^{2u}d\mu +\mu \left( \Omega \right) \geq \frac{4\pi }{%
\lambda }.
\end{equation*}%
In other words 
\begin{equation*}
\int_{\Omega }e^{2u_{1}}dx+\int_{\Omega }e^{2u_{2}}dx\geq \frac{4\pi }{%
\lambda }.
\end{equation*}
\end{proof}

By exactly the same argument as above, Theorem \ref{thm1.2} follows from
Theorem \ref{thm1.4}.

\begin{example}
Fix $0<r<1$. We take the stereographic projection of the unit sphere $S^{2}$
with respect the north pole to plane%
\begin{equation*}
x_{3}=-\sqrt{1-r^{2}}=-h_{1},
\end{equation*}%
then the standard metric on $S^{2}$ is written as%
\begin{equation*}
g_{1}=\frac{4\left( 1+h_{1}\right) ^{2}}{\left( \left\vert x\right\vert
^{2}+\left( 1+h_{1}\right) ^{2}\right) ^{2}}\left\vert dx\right\vert
^{2}=e^{2u_{1}}\left\vert dx\right\vert ^{2}.
\end{equation*}%
For $R>1$, we do stereographic projection of $R\cdot S^{2}$ with respect to
the north pole to the plane 
\begin{equation*}
x_{3}=\sqrt{R^{2}-r^{2}}=h_{2},
\end{equation*}%
then the metric on $R\cdot S^{2}$%
\begin{equation*}
g_{2}=\frac{4R^{2}\left( R-h_{2}\right) ^{2}}{\left( \left\vert x\right\vert
^{2}+\left( R-h_{2}\right) ^{2}\right) ^{2}}\left\vert dx\right\vert
^{2}=e^{2u_{2}}\left\vert dx\right\vert ^{2}.
\end{equation*}%
Note that for $\left\vert x\right\vert <r$, $u_{2}\left( x\right)
>u_{1}\left( x\right) $,%
\begin{eqnarray*}
\Delta u_{1}+e^{2u_{1}} &=&0, \\
\Delta u_{2}+R^{-2} e^{2u_{2}} &=&0, \\
\int_{B_{r}}e^{2u_{1}}dx &=&2\pi \left( 1-h_{1}\right) , \\
\int_{B_{r}}e^{2u_{2}}dx &=&2\pi R\left( R+h_{2}\right).
\end{eqnarray*}%
We have
\[\int_{B_r}e^{2u_1}dx+\int_{B_r}e^{2u_2}dx>4\pi R^2.\]
This is an example for Theorem \ref{thm1.1variant} with $\lambda =R^{-2}$.
\end{example}

\begin{example}
For $0<r<R<1$, we take the stereographic projection of $S^{2}$ with respect
to the north pole to the plane%
\begin{equation*}
x_{3}=\sqrt{1-r^{2}}=h_{1},
\end{equation*}%
then the metric on $S^{2}$ is written as%
\begin{equation*}
g_{1}=\frac{4\left( 1-h_{1}\right) ^{2}}{\left( \left\vert x\right\vert
^{2}+\left( 1-h_{1}\right) ^{2}\right) ^{2}}\left\vert dx\right\vert
^{2}=e^{2u_{1}}\left\vert dx\right\vert ^{2}.
\end{equation*}%
We also do stereographic projection of $R\cdot S^{2}$ with respect to the
north pole to the plane 
\begin{equation*}
x_{3}=-\sqrt{R^{2}-r^{2}}=-h_{2},
\end{equation*}%
then the metric on $R\cdot S^{2}$ is written as%
\begin{equation*}
g_{2}=\frac{4R^{2}\left( R+h_{2}\right) ^{2}}{\left( \left\vert x\right\vert
^{2}+\left( R+h_{2}\right) ^{2}\right) ^{2}}\left\vert dx\right\vert
^{2}=e^{2u_{2}}\left\vert dx\right\vert ^{2}.
\end{equation*}%
Note that for $\left\vert x\right\vert <r$, $u_{2}\left( x\right)
<u_{1}\left( x\right) $,%
\begin{eqnarray*}
\Delta u_{1}+e^{2u_{1}} &=&0, \\
\Delta u_{2}+e^{2u_{2}} &=&-\left( R^{-2}-1\right) e^{2u_{2}}, \\
\int_{B_{r}}e^{2u_{1}}dx &=&2\pi \left( 1+h_{1}\right) , \\
\int_{B_{r}}e^{2u_{2}}dx &=&2\pi R\left( R-h_{2}\right) >2\pi \left(
1-h_{1}\right) .
\end{eqnarray*}%
Hence
\[\int_{B_r}e^{2u_1}dx+\int_{B_r}e^{2u_2}dx>4\pi.\]
This is an example of Theorem \ref{thm1.2}.
\end{example}

\begin{example}
Let $\Omega \subset \mathbb{R}^{2}$ be a bounded smooth domain, $u$ be a
smooth function on $\overline{\Omega }$ such that%
\begin{eqnarray*}
-\Delta u+1 &\leq &e^{2u}\text{ in }\Omega , \\
\left. u\right\vert _{\partial \Omega } &=&0, \\
u &>&0\text{ in }\Omega.
\end{eqnarray*}%
It follows from Theorem \ref{thm1.3} that%
\begin{equation}
\int_{\Omega }e^{2u}dx+\left\vert \Omega \right\vert \geq 4\pi .
\end{equation}%
Because of the usual isoperimetric inequality on $\mathbb{R}^{2}$, the
assumption $\mu \left( M\right) \leq 4\pi $ in Theorem \ref{thm1.3} is not
needed in our situation. Here we will give an example where $\int_{\Omega
}e^{2u}dx+\left\vert \Omega \right\vert $ is arbitrary close to $4\pi $.

For $0<h<1$, denote $r=\sqrt{1-h^{2}}$. Take the stereographic projection of 
$S^{2}$ with respect to the north pole to the plane $x_{3}=-h$, then the
metric on $S^{2}$ is written as%
\begin{equation*}
g_{1}=\frac{4\left( 1+h\right) ^{2}}{\left[ \left\vert x\right\vert
^{2}+\left( 1+h\right) ^{2}\right] ^{2}}\left\vert dx\right\vert
^{2}=e^{2u_{1}}\left\vert dx\right\vert ^{2}.
\end{equation*}%
We have%
\begin{eqnarray*}
-\Delta u_{1} &=&e^{2u_{1}}\text{ in }B_{r}; \\
\left. u_{1}\right\vert _{\partial B_{r}} &=&0; \\
1 &\leq &e^{2u_{1}}\leq \frac{4}{\left( 1+h\right) ^{2}}.
\end{eqnarray*}%
Let%
\begin{equation*}
R=\frac{2}{1+h}
\end{equation*}%
and%
\begin{equation*}
h_{2}=\sqrt{R^{2}-r^{2}}.
\end{equation*}%
Take the stereographic projection of $R\cdot S^{2}$ with respect to the
north pole to the plane $x_{3}=h_{2}$, then the metric on $R\cdot S^{2}$ is
written as%
\begin{equation*}
g_{2}=\frac{4R^{2}\left( R-h_{2}\right) ^{2}}{\left[ \left\vert x\right\vert
^{2}+\left( R-h_{2}\right) ^{2}\right] ^{2}}\left\vert dx\right\vert
^{2}=e^{2u_{2}}\left\vert dx\right\vert ^{2}.
\end{equation*}%
We have%
\begin{eqnarray*}
-\Delta u_{2} &=&\frac{\left( 1+h\right) ^{2}}{4}e^{2u_{2}}\text{ in }B_{r};
\\
\left. u_{2}\right\vert _{\partial B_{r}} &=&0; \\
u_{2} &>&u_{1}\text{ in }B_{r}.
\end{eqnarray*}%
Let $u=u_{2}-u_{1}$, then $u>0$ in $B_{r}$, $\left. u\right\vert _{\partial
B_{r}}=0$. Moreover%
\begin{eqnarray*}
-\Delta u &=&\frac{\left( 1+h\right) ^{2}}{4}e^{2u_{2}}-e^{2u_{1}} \\
&=&\frac{\left( 1+h\right) ^{2}}{4}e^{2u_{1}}e^{2u}-e^{2u_{1}}.
\end{eqnarray*}%
It follows that%
\begin{eqnarray*}
-\Delta u+1 &\leq &-\Delta u+e^{2u_{1}} \\
&=&\frac{\left( 1+h\right) ^{2}}{4}e^{2u_{1}}e^{2u} \\
&\leq &e^{2u}.
\end{eqnarray*}%
On the other hand,%
\begin{equation*}
\int_{B_{r}}e^{2u}dx+\left\vert B_{r}\right\vert
=\int_{B_{r}}e^{2u_{2}-2u_{1}}dx+\left\vert B_{r}\right\vert \rightarrow 4\pi
\end{equation*}%
as $h\uparrow 1^{-}$. 

The above example shows that one can not get any improvements to  (\ref{eqmain}) by assuming $K\leq a<1$. Indeed $a=0$ in the example above. 
\end{example}

\section{Differential equation on surface satisfying general isoperimetric
inequalities\label{sec3}}

In this section we present sphere covering type inequalities on surfaces satisfying general isoperimetric inequalities (see (\ref{eqisop}) below) and  
discuss their applications to elliptic equations with exponential
nonlinearities. We find the following definition particularly useful. 

\begin{definition}
\label{def3.1}Let $M$ be a smooth surface and $g$ be a metric on $M$. If for
some $0<\theta \leq 1$ and $\kappa \in \mathbb{R}$, we have 
\begin{equation}
4\pi \theta \mu \left( E\right) -\kappa \mu^2\left( E\right)\leq s^{2} \left(
\partial E\right)  \label{eqisop}
\end{equation}%
for any compact smooth domain $E\subset M$, here $s$ is the one dimensional
measure associated with $g$ and $\mu $ is the two dimensional measure, then
we say $\left( M,g\right) $ satisfies the $\left( \theta ,\kappa \right) $%
-isoperimetric inequality.
\end{definition}

\begin{theorem}
\label{thm3.1} Let $\left( M,g\right) $ be a smooth Riemann surface
satisfying the $\left( \theta ,\kappa \right) $-isoperimetric inequality for
some $\theta \in \left( 0,1\right] $ and $\kappa \in \mathbb{R}$. If $\Omega
\subset M$ is an open domain with compact closure and 
$u\in C^{\infty }\left( \overline{\Omega }%
\right) $ such that 
\begin{equation}
-\Delta_g u+\kappa =\lambda e^{2u}+f,\quad \left. u\right\vert _{\partial
\Omega }=0,\quad u>0\text{ in }\Omega .
\end{equation}%
Denote%
\begin{equation}
\Theta =\frac{1}{2\pi }\int_{\Omega }f^{+}d\mu ,
\end{equation}%
then%
\begin{equation}
4\pi \left( \theta -\Theta \right) \int_{\Omega }e^{2u}d\mu -\lambda \left(
\int_{\Omega }e^{2u}d\mu \right) ^{2}\leq 4\pi \theta \mu \left( \Omega
\right) -\kappa \mu ^{2}\left( \Omega \right).
\end{equation}%
In particular, if $\Theta =0$ and $0<\lambda \leq \kappa $, then%
\begin{equation}
\int_{\Omega }e^{2u}d\mu +\mu \left( \Omega \right) \geq \frac{4\pi \theta }{%
\lambda }.
\end{equation}
\end{theorem}

\begin{remark}
\label{rmk3.1}It is worth pointing out that as long as the $\left( \theta
,\kappa \right) $-isoperimetric inequality is valid, the smoothness of $\ u$
and metric $g$ is not essential to our argument. In particular $f$ can be
replaced by a signed measure. This is useful in some singular Liouville type
equations. We will not elaborate this point further but refer the reader to 
\cite{BC1,BC2,BGJM} and the references therein.
\end{remark}

\begin{proof}
By approximation we can assume $u$ is a Morse function. For $t>0$, let%
\begin{equation*}
\alpha \left( t\right) =\int_{\{ u>t\}}e^{2u}d\mu ,\quad \beta \left( t\right)
=\int_{\{ u>t\}}d\mu .
\end{equation*}%
As in the proof of Theorem \ref{thm1.3}, we have%
\begin{eqnarray*}
\alpha ^{\prime }\left( t\right) &=&-e^{2t}\int_{\{ u=t\}}\frac{ds}{\left\vert
\nabla u\right\vert }, \\
\beta ^{\prime }\left( t\right) &=&-\int_{\{ u=t\}}\frac{ds}{\left\vert \nabla
u\right\vert }, \\
\alpha ^{\prime }\left( t\right) &=&e^{2t}\beta ^{\prime }\left( t\right) .
\end{eqnarray*}%
On the other hand,%
\begin{equation*}
\int_{\{ u>t\}}\left( -\Delta u\right) d\mu +\kappa \beta \left( t\right)
=\lambda \alpha \left( t\right) +\int_{\{ u>t\}}fd\mu ,
\end{equation*}%
hence%
\begin{equation*}
\int_{\{ u=t\}}\left\vert \nabla u\right\vert ds\leq \lambda \alpha -\kappa \beta
+2\pi \Theta .
\end{equation*}%
We have%
\begin{equation*}
e^{2t}\int_{\{ u=t\}}\frac{ds}{\left\vert \nabla u\right\vert }%
\int_{\{ u=t\}}\left\vert \nabla u\right\vert ds\leq -\lambda \alpha \alpha
^{\prime }+\kappa e^{2t}\beta \beta ^{\prime }-2\pi \Theta \alpha ^{\prime },
\end{equation*}%
which implies%
\begin{equation*}
e^{2t}s^{2} \left( \left\{ u=t\right\} \right) \leq -\lambda \alpha \alpha
^{\prime }+\kappa e^{2t}\beta \beta ^{\prime }-2\pi \Theta \alpha ^{\prime }.
\end{equation*}%
Using (\ref{eqisop}) we get 
\begin{equation*}
e^{2t}\left( 4\pi \theta \beta -\kappa \beta ^{2}\right) \leq -\lambda
\alpha \alpha ^{\prime }+\kappa e^{2t}\beta \beta ^{\prime }-2\pi \Theta
\alpha ^{\prime }.
\end{equation*}%
It follows that%
\begin{equation*}
4\pi \theta \left[ \left( e^{2t}\beta \right) ^{\prime }-\alpha ^{\prime }%
\right] \leq -\lambda \left( \alpha ^{2}\right) ^{\prime }+\kappa \left(
e^{2t}\beta ^{2}\right) ^{\prime }-4\pi \Theta \alpha ^{\prime }.
\end{equation*}%
Integrating for $t$ from $0$ to $\infty $, we get%
\begin{equation*}
4\pi \theta \left( \alpha \left( 0\right) -\beta \left( 0\right) \right)
\leq \lambda \alpha \left( 0\right) ^{2}-\kappa \beta \left( 0\right)
^{2}+4\pi \Theta \alpha \left( 0\right)
\end{equation*}%
In other words 
\begin{equation*}
4\pi \left( \theta -\Theta \right) \int_{\Omega }e^{2u}d\mu -\lambda \left(
\int_{\Omega }e^{2u}d\mu \right) ^{2}\leq 4\pi \theta \mu ^{2}\left( \Omega
\right) -\kappa \mu \left( \Omega \right).
\end{equation*}
\end{proof}

If we flip the inequalities as in the proof of Theorem \ref{thm1.4}, we get

\begin{theorem}
\label{thm3.2}Let $\left( M,g\right) $ be a smooth Riemann surface
satisfying the $\left( \theta ,\kappa \right) $-isoperimetric inequality for
some $\theta \in \left( 0,1\right] $ and $\kappa \in \mathbb{R}$. Assume $%
\Omega \subset M$ is an open domain with compact closure and 
$u\in C^{\infty }\left( \overline{%
\Omega }\right) $ such that 
\begin{equation}
-\Delta_g u+\kappa =\lambda e^{2u}+f,\quad \left. u\right\vert _{\partial
\Omega }=0,\quad u<0\text{ in }\Omega .
\end{equation}%
Denote%
\begin{equation}
\Theta =\frac{1}{2\pi }\int_{\Omega }f^{-}d\mu ,
\end{equation}%
then%
\begin{equation}
4\pi \theta \mu \left( \Omega \right) -\kappa \mu ^{2} \left( \Omega \right) \leq 4\pi \left( \theta +\Theta \right) \int_{\Omega }e^{2u}d\mu
-\lambda \left( \int_{\Omega }e^{2u}d\mu \right) ^{2}.
\end{equation}%
In particular, if $\Theta =0$ and $\lambda =\kappa >0$, then%
\begin{equation}
\int_{\Omega }e^{2u}d\mu +\mu \left( \Omega \right) \geq \frac{4\pi \theta }{%
\lambda }.
\end{equation}
\end{theorem}

Next we discuss some known and new applications of Theorem \ref%
{thm3.1} and \ref{thm3.2}.

\begin{example}[\protect\cite{B,S}]
\label{ex3.1}Let $\left( M,\widetilde{g}\right) $ be a simply connected
smooth Riemann surface with curvature $\widetilde{K}\leq 1$. If $E$ is a
compact simply connected domain in $M$ with nonempty smooth boundary, then
we can find $u\in C^{\infty }\left( E\right) $ such that%
\begin{equation*}
-{\Delta }_{\widetilde g} u=\widetilde{K}\text{ on }E,\quad \left. u\right\vert
_{\partial E}=0.
\end{equation*}%
Let $g=e^{-2u}\widetilde{g}$, then the curvature of $g$ is zero. By Riemann
mapping theorem and the Taylor series argument for holomorphic functions in 
\cite{C}, $\left( E,g\right) $ satisfies the $\left( 1,0\right) $%
-isoperimetric inequality. On the other hand,%
\begin{equation*}
-\Delta _{g}u\leq e^{2u}\text{ on }E,\quad \left. u\right\vert _{\partial
E}=0.
\end{equation*}%
If we let%
\begin{equation*}
\Omega =\left\{ p\in E:u\left( p\right) >0\right\} ,
\end{equation*}%
then%
\begin{equation*}
-\Delta _{g}u\leq e^{2u}\text{ on }\Omega ,\quad \left. u\right\vert
_{\partial \Omega }=0,\quad u>0\text{ in }\Omega .
\end{equation*}%
Theorem \ref{thm3.1} tells us%
\begin{equation*}
4\pi \int_{\Omega }e^{2u}d\mu -\left( \int_{\Omega }e^{2u}d\mu \right)
^{2}\leq 4\pi \mu \left( \Omega \right) .
\end{equation*}%
Hence%
\begin{eqnarray*}
4\pi \left( \int_{E}e^{2u}d\mu -\mu \left( E\right) \right)  &\leq &4\pi
\left( \int_{\Omega }e^{2u}d\mu -\mu \left( \Omega \right) \right)  \\
&\leq &\left( \int_{\Omega }e^{2u}d\mu \right) ^{2} \\
&\leq &\left( \int_{E}e^{2u}d\mu \right) ^{2}
\end{eqnarray*}%
i.e.%
\begin{equation*}
4\pi \widetilde{\mu }\left( E\right) -\widetilde{\mu }^{2}\left( E\right)
\leq 4\pi \mu \left( E\right) \leq s ^{2}\left( \partial E\right)=%
\widetilde{s} ^{2}\left( \partial E\right).
\end{equation*}%
This is exactly the argument given in \cite{B,S}.

If we assume further that $\widetilde{\mu }\left( M\right) \leq 4\pi $,  then following \cite[Lemma 4.2]{CCL}  we know,   for $E$ to be a compact domain with  boundary  smooth but not necessarily simply connected, there still holds%
\begin{equation*}
4\pi \widetilde{\mu }\left( E\right) -\widetilde{\mu } ^{2}\left( E\right) \leq \widetilde{s} ^{2}\left( \partial E\right).
\end{equation*}%
In another word, $\left( 1,1\right) $-isoperimetric inequality is true for $%
\left( M,\widetilde{g}\right) $. As a consequence Theorem \ref{thm1.3}
follows from Theorem \ref{thm3.1}.
\end{example}

\begin{example}
\label{ex3.2} Let $\left( M,\widetilde{g}\right) $ be a simply connected
smooth Riemann surface with curvature $\tilde{K}$. Assume $a\geq 0$ and%
\begin{equation}
\Theta =\frac{1}{2\pi }\int_{M}\left( \widetilde{K}-a\right) ^{+}d\widetilde{%
\mu }<1.
\end{equation}%
Then for any compact simply connected domain $E$ in $M$ with nonempty smooth
boundary, we have%
\begin{equation}
4\pi \left( 1-\Theta \right) \widetilde{\mu }\left( E\right) -a\widetilde{\mu 
} ^{2}\left( E\right)\leq \widetilde{s} ^{2}\left( \partial E\right).
\end{equation}%
In particular, if we assume further that $\widetilde{\mu }\left( M\right) \leq
\frac{4\pi \left( 1-\Theta \right) }{a}$, then $\left( M,\widetilde{g}\right) $
satisfies $\left( 1-\Theta ,a\right) $-isoperimetric inequality. In fact
this is even true when $\widetilde{g}$ is singular, see \cite{BC1,BC2,BGJM}
and the references therein.

Indeed as in the previous example, we can find $u\in C^{\infty }\left(
E\right) $ such that%
\begin{equation*}
-\widetilde{\Delta }u=\widetilde{K}\text{ on }E,\quad \left. u\right\vert
_{\partial E}=0.
\end{equation*}%
Let $g=e^{-2u}\widetilde{g}$, then $\left( E,g\right) $ satisfies $\left(
1,0\right) $-isoperimetric inequality and%
\begin{equation*}
-\Delta _{g}u=ae^{2u}+\left( \widetilde{K}-a\right) e^{2u}\text{ on }E.
\end{equation*}%
Note that%
\begin{equation*}
\frac{1}{2\pi }\int_{E}\left[ \left( \widetilde{K}-a\right) e^{2u}\right]
^{+}d\mu =\frac{1}{2\pi }\int_{E}\left( \widetilde{K}-a\right) ^{+}d%
\widetilde{\mu }\leq \Theta <1.
\end{equation*}%
Let 
\begin{equation*}
\Omega =\left\{ p\in E:u\left( p\right) >0\right\} ,
\end{equation*}%
then%
\begin{equation*}
-\Delta _{g}u=ae^{2u}+\left( \widetilde{K}-a\right) e^{2u}\text{ on }\Omega
,\quad \left. u\right\vert _{\partial \Omega }=0,\quad u>0\text{ in }\Omega .
\end{equation*}%
Theorem \ref{thm3.1} implies%
\begin{equation*}
4\pi \left( 1-\Theta \right) \int_{\Omega }e^{2u}d\mu -a\left( \int_{\Omega
}e^{2u}d\mu \right) ^{2}\leq 4\pi \mu \left( \Omega \right) .
\end{equation*}%
Hence%
\begin{eqnarray*}
&&4\pi \left( 1-\Theta \right) \int_{E}e^{2u}d\mu -4\pi \mu \left( E\right)
\\
&\leq &4\pi \left( 1-\Theta \right) \int_{\Omega }e^{2u}d\mu -4\pi \mu
\left( \Omega \right) \\
&\leq &a\left( \int_{\Omega }e^{2u}d\mu \right) ^{2} \\
&\leq &a\left( \int_{E}e^{2u}d\mu \right) ^{2}.
\end{eqnarray*}%
In another word,%
\begin{equation*}
4\pi \left( 1-\Theta \right) \widetilde{\mu }\left( E\right) -a\widetilde{\mu 
} ^{2} \left( E\right) \leq 4\pi \mu \left( E\right) \leq s ^{2}\left( \partial
E\right) =\widetilde{s} ^{2}\left( \partial E\right) .
\end{equation*}
\end{example}

\begin{example}[\protect\cite{B,BC1,BC2,BGJM}]
\label{ex3.3}Let $\Omega _{0}\subset \mathbb{R}^{2}$ be a simply connected
domain and $u,h\in C^{\infty }\left( \Omega _{0}\right) $ with $h\left(
x\right) >0$ for any $x\in \Omega _{0}$. We write%
\begin{equation}
-\Delta u=he^{2u}+f
\end{equation}%
and%
\begin{equation}
\Theta =\frac{1}{2\pi }\int_{\Omega _{0}}\left( f-\frac{1}{2}\Delta \log
h\right) ^{+}dx.
\end{equation}%
If $\Theta <1$ and%
\begin{equation}
\int_{\Omega _{0}}he^{2u}dx\leq 4\pi \left( 1-\Theta \right) ,
\end{equation}%
then $\left( \Omega _{0},he^{2u}\left\vert dx\right\vert ^{2}\right) $
satisfies the $\left( 1-\Theta ,1\right) $-isoperimetric inequality. As
pointed out earlier in Remark \ref{rmk3.1}, the regularity assumption of $u$
and $h$ can be weakened and we refer the reader to \cite{BC1,BC2,BGJM}.
\end{example}

\begin{proof}
For convenience we denote $g=he^{2u}\left\vert dx\right\vert ^{2}$, then its
curvature%
\begin{eqnarray*}
K &=&h^{-1}e^{-2u}\left( -\Delta u-\frac{1}{2}\Delta \log h\right) \\
&=&h^{-1}e^{-2u}\left( he^{2u}+f-\frac{1}{2}\Delta \log h\right) \\
&=&1+h^{-1}e^{-2u}\left( f-\frac{1}{2}\Delta \log h\right) .
\end{eqnarray*}%
Hence%
\begin{equation*}
\frac{1}{2\pi }\int_{\Omega _{0}}\left( K-1\right) ^{+}d\mu =\frac{1}{2\pi }%
\int_{\Omega _{0}}\left( f-\frac{1}{2}\Delta \log h\right) ^{+}dx=\Theta <1,
\end{equation*}%
and%
\begin{equation*}
\mu \left( \Omega _{0}\right) =\int_{\Omega _{0}}he^{2u}dx\leq 4\pi \left(
1-\Theta \right) .
\end{equation*}%
It follows from Example \ref{ex3.2} that $\left( \Omega _{0},g\right) $
satisfies the $\left( 1-\Theta ,1\right) $-isoperimetric inequality
\end{proof}

If we replace the reference metric from Euclidean metric to an arbitrary
one, we end up with the following formulation.

\begin{lemma}
\label{lem3.1}Let $\left( M,g\right) $ be a simply connected Riemann surface with curvature $K$,
and $u\in C^{\infty }\left( M\right) $. We write%
\begin{equation}
-\Delta _{g}u=e^{2u}+f
\end{equation}%
and%
\begin{equation}
\Theta =\frac{1}{2\pi }\int_{M}\left( f+K\right) ^{+}d\mu .
\end{equation}%
If $\Theta <1$ and%
\begin{equation}
\int_{M}e^{2u}d\mu \leq 4\pi \left( 1-\Theta \right) ,
\end{equation}%
then $\left( M,e^{2u}g\right) $ satisfies the $\left( 1-\Theta ,1\right) $%
-isoperimetric inequality.
\end{lemma}

\begin{proof}
Let $\widetilde{g}=e^{2u}g$, then%
\begin{equation*}
\widetilde{K}=e^{-2u}\left( K-\Delta u\right) =1+e^{-2u}\left( f+K\right) .
\end{equation*}%
In particular,%
\begin{equation*}
\frac{1}{2\pi }\int_{M}\left( \widetilde{K}-1\right) ^{+}d\widetilde{\mu }=%
\frac{1}{2\pi }\int_{M}\left( f+K\right) ^{+}d\mu =\Theta <1,
\end{equation*}%
and%
\begin{equation*}
\widetilde{\mu }\left( M\right) =\int_{M}e^{2u}d\mu \leq 4\pi \left(
1-\Theta \right) .
\end{equation*}%
It follows from Example \ref{ex3.2} that $\left( M,\widetilde{g}\right) $
satisfies the $\left( 1-\Theta ,1\right) $-isoperimetric inequality
\end{proof}

Note that Lemma \ref{lem3.1} also follows form Example \ref{ex3.3} and
Riemann mapping theorem. With Lemma \ref{lem3.1} at hand, we can deduce
easily a variation of sphere covering inequality.

\begin{proposition}
\label{prop3.1}Let $\left( M,g\right) $ be a simply connected Riemann
surface, and $u_{1}\in C^{\infty }\left( M\right) $. We write%
\begin{equation}
-\Delta _{g}u_{1}=e^{2u_{1}}+f
\end{equation}%
and%
\begin{equation}
\Theta =\frac{1}{2\pi }\int_{M}\left( f+K\right) ^{+}d\mu .
\end{equation}%
Here $K$ is the curvature of $\left( M,g\right)$. Assume $\Theta <1$ and%
\begin{equation}
\int_{M}e^{2u_{1}}d\mu \leq 4\pi \left( 1-\Theta \right) .
\end{equation}%
Let $\Omega \subset M$ be a domain with a compact closure and nonempty
boundary. Assume $u_{2}\in C^{\infty }\left( \overline{\Omega }\right) $ and 
$0<\lambda \leq 1$ such that%
\begin{eqnarray*}
\Delta _{g}u_{2}+\lambda e^{2u_{2}} &\geq &\Delta _{g}u_{1}+e^{2u_{1}}%
\text{ in }\Omega \text{,} \\
u_{2} &>&u_{1}\text{ in }\Omega , \\
\left. u_{2}\right\vert _{\partial \Omega } &=&\left. u_{1}\right\vert
_{\partial \Omega }.
\end{eqnarray*}%
Then%
\begin{equation}
\int_{\Omega }e^{2u_{1}}d\mu +\int_{\Omega }e^{2u_{2}}d\mu \geq \frac{4\pi
\left( 1-\Theta \right) }{\lambda }.
\end{equation}
\end{proposition}

Note that Theorem \ref{thm1.1variant} is a special case of Proposition \ref%
{prop3.1}.

\begin{proof}[Proof of Proposition \protect\ref{prop3.1}]
Let $\widetilde{g}=e^{2u_{1}}g$, then it follows from Lemma \ref{lem3.1}
that $\left( M,\widetilde{g}\right) $ satisfies the $\left( 1-\Theta
,1\right) $-isoperimetric inequality. Let $u=u_{2}-u_{1}$, then on $\Omega $
we have%
\begin{equation*}
-\Delta _{g}u\leq e^{2u_{1}}\left( \lambda e^{2u}-1\right) .
\end{equation*}%
Hence%
\begin{equation*}
-\Delta _{\widetilde{g}}u\leq \lambda e^{2u}-1.
\end{equation*}%
Moreover $u>0$ in $\Omega $ and $\left. u\right\vert _{\partial \Omega }=0$,
it follows from Theorem \ref{thm3.1} that%
\begin{equation*}
\int_{\Omega }e^{2u}d\widetilde{\mu }+\widetilde{\mu }\left( \Omega \right)
\geq \frac{4\pi \left( 1-\Theta \right) }{\lambda }.
\end{equation*}%
In another word,%
\begin{equation*}
\int_{\Omega }e^{2u_{1}}d\mu +\int_{\Omega }e^{2u_{2}}d\mu \geq \frac{4\pi
\left( 1-\Theta \right) }{\lambda }.
\end{equation*}
\end{proof}

Using Theorem \ref{thm3.2}, with the same proof, we also have a dual
inequality generalizing Theorem \ref{thm1.2}.

\begin{proposition}
\label{prop3.2}Let $\left( M,g\right) $ be a simply connected Riemann
surface with curvature $K$, and $u_{1}\in C^{\infty }\left( M\right) $. We write%
\begin{equation}
-\Delta _{g}u_{1}=e^{2u_{1}}+f
\end{equation}%
and%
\begin{equation}
\Theta =\frac{1}{2\pi }\int_{M}\left( f+K\right) ^{+}d\mu .
\end{equation}%
Assume $\Theta <1$ and%
\begin{equation}
\int_{M}e^{2u_{1}}d\mu \leq 4\pi \left( 1-\Theta \right) .
\end{equation}%
Let $\Omega \subset M$ be a domain with compact closure and nonempty
boundary. Assume $u_{2}\in C^{\infty }\left( \overline{\Omega }\right) $
such that%
\begin{eqnarray*}
\Delta _{g}u_{2}+e^{2u_{2}} &\leq &\Delta _{g}u_{1}+e^{2u_{1}}\text{ in }%
\Omega \text{,} \\
u_{2} &<&u_{1}\text{ in }\Omega , \\
\left. u_{2}\right\vert _{\partial \Omega } &=&\left. u_{1}\right\vert
_{\partial \Omega }.
\end{eqnarray*}%
Then%
\begin{equation}
\int_{\Omega }e^{2u_{1}}d\mu +\int_{\Omega }e^{2u_{2}}d\mu \geq 4\pi \left(
1-\Theta \right) .
\end{equation}
\end{proposition}

\begin{example}
\label{ex3.4}Let $\left( M,g\right) $ be a simply connected Riemann surface with curvature $K$,
and $u_{1},h\in C^{\infty }\left( M\right) $ with $h>0$. We write%
\begin{equation}
-\Delta _{g}u_{1}=he^{2u_{1}}+f
\end{equation}%
and%
\begin{equation}
\Theta =\frac{1}{2\pi }\int_{M}\left( f+K-\frac{1}{2}\Delta _{g}\log
h\right) ^{+}d\mu .
\end{equation}%
Assume $\Theta <1$ and%
\begin{equation}
\int_{M}he^{2u_{1}}d\mu \leq 4\pi \left( 1-\Theta \right) .
\end{equation}%
Let $\Omega \subset M$ be a domain with compact closure and nonempty
boundary. Assume $u_{2}\in C^{\infty }\left( \overline{\Omega }\right) $ and 
$0<\lambda \leq 1$ such that%
\begin{eqnarray*}
\Delta _{g}u_{2}+\lambda he^{2u_{2}} &\geq &\Delta _{g}u_{1}+he^{2u_{1}}%
\text{ in }\Omega \text{,} \\
u_{2} &>&u_{1}\text{ in }\Omega , \\
\left. u_{2}\right\vert _{\partial \Omega } &=&\left. u_{1}\right\vert
_{\partial \Omega }.
\end{eqnarray*}%
Then%
\begin{equation}
\int_{\Omega }he^{2u_{1}}d\mu +\int_{\Omega }he^{2u_{2}}d\mu \geq \frac{4\pi
\left( 1-\Theta \right) }{\lambda }.
\end{equation}
\end{example}

\begin{proof}
Let%
\begin{equation*}
v_{1}=u_{1}+\frac{1}{2}\log h,\quad v_{2}=u_{2}+\frac{1}{2}\log h,
\end{equation*}%
then%
\begin{equation}
-\Delta _{g}v_{1}=e^{2v_{1}}+f-\frac{1}{2}\Delta \log h.
\end{equation}%
Moreover%
\begin{eqnarray*}
\Delta _{g}v_{2}+\lambda e^{2v_{2}} &\geq &\Delta _{g}v_{1}+e^{2v_{1}}%
\text{ in }\Omega \text{,} \\
v_{2} &>&v_{1}\text{ in }\Omega , \\
\left. v_{2}\right\vert _{\partial \Omega } &=&\left. v_{1}\right\vert
_{\partial \Omega }.
\end{eqnarray*}%
Then we can apply Proposition \ref{prop3.1} to get the desired conclusion.
\end{proof}

By a straightforward modification we can also deal with the case $h$ changes
sign.

\begin{example}
\label{ex3.5}Let $\left( M,g\right) $ be a simply connected Riemann surface with curvature $K$,
and $u_{1},h,H\in C^{\infty }\left( M\right) $ with $h\leq H$ and $H>0$. We
write%
\begin{equation}
-\Delta _{g}u_{1}=He^{2u_{1}}+f
\end{equation}%
and%
\begin{equation}
\Theta =\frac{1}{2\pi }\int_{M}\left( f+K-\frac{1}{2}\Delta _{g}\log
H\right) ^{+}d\mu .
\end{equation}%
Assume $\Theta <1$ and%
\begin{equation}
\int_{M}He^{2u_{1}}d\mu \leq 4\pi \left( 1-\Theta \right) .
\end{equation}%
Let $\Omega \subset M$ be a domain with compact closure and nonempty
boundary. Assume $u_{2}\in C^{\infty }\left( \overline{\Omega }\right) $
such that%
\begin{eqnarray*}
\Delta _{g}u_{2}+he^{2u_{2}} &\geq &\Delta _{g}u_{1}+he^{2u_{1}}\text{ in }%
\Omega \text{,} \\
u_{2} &>&u_{1}\text{ in }\Omega , \\
\left. u_{2}\right\vert _{\partial \Omega } &=&\left. u_{1}\right\vert
_{\partial \Omega }.
\end{eqnarray*}%
Then%
\begin{equation}
\int_{\Omega }He^{2u_{1}}d\mu +\int_{\Omega }He^{2u_{2}}d\mu \geq 4\pi
\left( 1-\Theta \right) .
\end{equation}
\end{example}

\begin{proof}
We have%
\begin{eqnarray*}
\Delta _{g}u_{2} &\geq &\Delta _{g}u_{1}-h\left(
e^{2u_{2}}-e^{2u_{1}}\right) \\
&\geq &\Delta _{g}u_{1}-H\left( e^{2u_{2}}-e^{2u_{1}}\right) .
\end{eqnarray*}%
Hence%
\begin{equation*}
\Delta _{g}u_{2}+He^{2u_{2}}\geq \Delta _{g}u_{1}+He^{2u_{1}}\text{ in }%
\Omega \text{.}
\end{equation*}%
Then we can apply Example \ref{ex3.4}.
\end{proof}

Next we turn to solutions of semilinear equations with equal weights, see \cite{BGJM, GM3}. 

\begin{proposition}
\label{prop3.3}Let $\Omega \subset \mathbb{R}^{2}$ be a bounded open simply
connected domain. Assume $u_{1},u_{2}\in C^{\infty }\left( \overline{\Omega }%
\right) $ such that%
\begin{equation}
\Delta u_{1}+e^{2u_{1}}\geq 0
\end{equation}%
and%
\begin{eqnarray*}
\Delta u_{1}+e^{2u_{1}} &\geq & \Delta u_{2}+e^{2u_{2}}\text{ in }\Omega ,
\\
u_{1}+c &>&u_{2}\text{ in }\Omega , \\
\left. u_{1}\right\vert _{\partial \Omega }+c &=&\left. u_{2}\right\vert
_{\partial \Omega }.
\end{eqnarray*}%
Here $c$ is a constant. If%
\begin{equation}
\int_{\Omega }e^{2u_{1}}dx=\int_{\Omega }e^{2u_{2}}dx=\rho ,
\end{equation}%
then $\rho \geq 4\pi $.
\end{proposition}

\begin{proof}
Note that $c>0$. If $\rho \leq 4\pi $, we will show $\rho =4\pi $. Indeed
let $g=e^{2u_{1}}\left\vert dx\right\vert ^{2}$, then $K\leq 1$ and $\mu
\left( \Omega \right) =\rho \leq 4\pi $. If we write $u=u_{2}-u_{1}-c$, then%
\begin{equation*}
-\Delta _{g}u+1\geq e^{2c}\cdot e^{2u}.
\end{equation*}%
Moreover $u<0$ in $\Omega $ and $\left. u\right\vert _{\partial \Omega }=0$.
It follows from Theorem \ref{thm1.4} that%
\begin{equation*}
4\pi \int_{\Omega }e^{2u}d\mu -e^{2c}\left( \int_{\Omega }e^{2u}d\mu \right)
^{2}\geq 4\pi \mu \left( \Omega \right) -\mu ^{2}\left( \Omega \right).
\end{equation*}%
Hence%
\begin{equation*}
0\geq \left( 1-e^{-2c}\right) \rho \left( 4\pi -\rho \right)
\end{equation*}%
and we get $\rho \geq 4\pi $.
\end{proof}

\begin{proposition}
\label{prop3.4}Let $\Omega \subset \mathbb{R}^{2}$ be a bounded open simply
connected domain. Assume $u_{1},u_{2}\in C^{\infty }\left( \overline{\Omega }%
\right) $ such that%
\begin{eqnarray*}
\Delta u_{1}+e^{2u_{1}} &=&\Delta u_{2}+e^{2u_{2}}\geq 0\text{ in }\Omega ,
\\
\left. u_{1}\right\vert _{\partial \Omega }+c &=&\left. u_{2}\right\vert
_{\partial \Omega }.
\end{eqnarray*}%
Here $c$ is a constant. If $u_{1}$ is not identically equal to $u_{2}$ and%
\begin{equation}
\int_{\Omega }e^{2u_{1}}dx=\int_{\Omega }e^{2u_{2}}dx=\rho ,
\end{equation}%
then $\rho \geq 4\pi $.
\end{proposition}

\begin{proof}
If $\rho \leq 4\pi $, we will show $\rho =4\pi $. Indeed let $%
g=e^{2u_{1}}\left\vert dx\right\vert ^{2}$, then $K\leq 1$ and $\mu \left(
\Omega \right) =\rho \leq 4\pi $. If we write $u=u_{2}-u_{1}-c$, then%
\begin{equation}
-\Delta _{g}u+1=e^{2c}\cdot e^{2u}.
\end{equation}%
Let%
\begin{eqnarray*}
\Omega ^{+} &=&\left\{ x\in \Omega :u\left( x\right) >0\right\} , \\
\Omega ^{-} &=&\left\{ x\in \Omega :u\left( x\right) <0\right\} ,
\end{eqnarray*}%
then it follows from unique continuation property that $\left\vert \Omega
\backslash \left( \Omega ^{+}\cup \Omega ^{-}\right) \right\vert =0$. On $%
\Omega ^{+}$, by Theorem \ref{thm1.3} we have%
\begin{equation*}
4\pi \int_{\Omega ^{+}}e^{2u}d\mu -e^{2c}\left( \int_{\Omega ^{+}}e^{2u}d\mu
\right) ^{2}\leq 4\pi \mu \left( \Omega ^{+}\right) -\mu  ^{2}\left( \Omega
^{+}\right).
\end{equation*}%
On $\Omega ^{-}$, by Theorem \ref{thm1.4} we have%
\begin{equation*}
4\pi \int_{\Omega ^{-}}e^{2u}d\mu -e^{2c}\left( \int_{\Omega ^{-}}e^{2u}d\mu
\right) ^{2}\geq 4\pi \mu \left( \Omega ^{-}\right) -\mu ^{2}\left( \Omega
^{-}\right).
\end{equation*}%
Using%
\begin{eqnarray*}
\mu \left( \Omega ^{+}\right) +\mu \left( \Omega ^{-}\right)  &=&\rho , \\
\int_{\Omega ^{+}}e^{2u}d\mu +\int_{\Omega ^{-}}e^{2u}d\mu  &=&e^{-2c}\rho ,
\end{eqnarray*}%
subtracting the two inequalities we get%
\begin{equation*}
\left( 4\pi -\rho \right) \left( \int_{\Omega ^{+}}e^{2u}d\mu -\int_{\Omega
^{-}}e^{2u}d\mu \right) \leq \left( 4\pi -\rho \right) \left( \mu \left(
\Omega ^{+}\right) -\mu \left( \Omega ^{-}\right) \right) .
\end{equation*}%
In another word%
\begin{equation*}
\left( 4\pi -\rho \right) \left( \int_{\Omega ^{+}}e^{2u}d\mu -\mu \left(
\Omega ^{+}\right) +\mu \left( \Omega ^{-}\right) -\int_{\Omega
^{-}}e^{2u}d\mu \right) \leq 0.
\end{equation*}%
Since $u$ is not identically equal to $0$, we see%
\begin{equation*}
\int_{\Omega ^{+}}e^{2u}d\mu -\mu \left( \Omega ^{+}\right) +\mu \left(
\Omega ^{-}\right) -\int_{\Omega ^{-}}e^{2u}d\mu >0\text{.}
\end{equation*}%
Hence $\rho \geq 4\pi $.
\end{proof}

We can replace the Euclidean domain with a Riemann surface.

\begin{example}
\label{ex3.6}Let $\left( M,g\right) $ be a simply connected compact Riemann
surface with nonempty boundary, and $u_{1}\in C^{\infty }\left( M\right) $.
We write%
\begin{equation}
-\Delta _{g}u_{1}=e^{2u_{1}}+f
\end{equation}%
and%
\begin{equation}
\Theta =\frac{1}{2\pi }\int_{M}\left( f+K\right) ^{+}d\mu .
\end{equation}%
Here $K$ is the curvature of $g$. Assume $u_{2}\in C^{\infty }\left(
M\right) $ such that%
\begin{eqnarray*}
\Delta _{g}u_{1}+e^{2u_{1}} &\geq &\Delta _{g}u_{2}+e^{2u_{2}}\text{ in }M%
\text{,} \\
u_{1}+c &>&u_{2}\text{ in }M, \\
\left. u_{1}\right\vert _{\partial M}+c &=&\left. u_{2}\right\vert
_{\partial M}.
\end{eqnarray*}%
Here $c$ is a constant. If 
\begin{equation}
\int_{M}e^{2u_{1}}d\mu =\int_{M}e^{2u_{1}}d\mu =\rho ,
\end{equation}%
then%
\begin{equation}
\rho \geq 4\pi \left( 1-\Theta \right) .
\end{equation}
\end{example}

\begin{proof}
Without loss of generality we can assume $\Theta <1$ and $\rho \leq 4\pi
\left( 1-\Theta \right) $. Let $\widetilde{g}=e^{2u_{1}}g$, then Lemma \ref%
{lem3.1} implies $\left( M,\widetilde{g}\right) $ satisfies $\left( 1-\Theta
,1\right) $-isoperimetric inequality. If we write $u=u_{2}-u_{1}-c$, then%
\begin{equation*}
-\widetilde{\Delta }u+1\geq e^{2c}\cdot e^{2u}\text{ on }M\text{.}
\end{equation*}%
Moreover $u<0$ in $M$ and $\left. u\right\vert _{\partial M}=0$. It follows
from Theorem \ref{thm3.2} that%
\begin{equation*}
4\pi \left( 1-\Theta \right) \int_{M}e^{2u}d\widetilde{\mu }-e^{2c}\left(
\int_{M}e^{2u}d\widetilde{\mu }\right) ^{2}\geq 4\pi \left( 1-\Theta \right) 
\widetilde{\mu }\left( M\right) -\widetilde{\mu } ^{2} \left( M\right).
\end{equation*}%
Hence%
\begin{equation*}
0\geq \left( 1-e^{-2c}\right) \rho \left( 4\pi \left( 1-\Theta \right) -\rho
\right) .
\end{equation*}%
Since $c>0$, we get $\rho \geq 4\pi \left( 1-\Theta \right) $.
\end{proof}

Using the argument in Example \ref{ex3.4} we get

\begin{example}
\label{ex3.7}Let $\left( M,g\right) $ be a simply connected compact Riemann
surface with nonempty boundary and curvature $K$, and $u_{1},h\in C^{\infty }\left( M\right) $
with $h>0$. We write%
\begin{equation}
-\Delta _{g}u_{1}=he^{2u_{1}}+f
\end{equation}%
and%
\begin{equation}
\Theta =\frac{1}{2\pi }\int_{M}\left( f+K-\frac{1}{2}\Delta _{g}\log
h\right) ^{+}d\mu .
\end{equation}%
Assume $u_{2}\in C^{\infty }\left( M\right) $ such that%
\begin{eqnarray*}
\Delta _{g}u_{1}+he^{2u_{1}} &\geq &\Delta _{g}u_{2}+he^{2u_{2}}\text{ in }%
M\text{,} \\
u_{1}+c &>&u_{2}\text{ in }M, \\
\left. u_{1}\right\vert _{\partial M}+c &=&\left. u_{2}\right\vert
_{\partial M}.
\end{eqnarray*}%
Here $c$ is a constant. If%
\begin{equation}
\int_{M}he^{2u_{1}}d\mu =\int_{M}he^{2u_{1}}d\mu =\rho ,
\end{equation}%
then%
\begin{equation}
\rho \geq 4\pi \left( 1-\Theta \right) .
\end{equation}
\end{example}

In the same spirit as the proof of Proposition \ref{prop3.4} but using both
Theorem \ref{thm3.1} and \ref{thm3.2} instead we have

\begin{example}
\label{ex3.8}Let $\left( M,g\right) $ be a simply connected compact Riemann
surface with nonempty boundary and curvature $K$, and $u_{1},u_{2},h\in C^{\infty }\left(
M\right) $ with $h>0$. Assume%
\begin{eqnarray}
-\Delta u_{1}-he^{2u_{1}} &=&-\Delta u_{2}-he^{2u_{2}}=f\text{ in }M, \\
\left. u_{1}\right\vert _{\partial M}+c &=&\left. u_{2}\right\vert
_{\partial M}.
\end{eqnarray}%
Here $c$ is a constant. We denote%
\begin{equation}
\Theta =\frac{1}{2\pi }\int_{M}\left( f+K-\frac{1}{2}\Delta _{g}\log
h\right) ^{+}d\mu .
\end{equation}%
If $u_{1}$ is not identically equal to $u_{2}$ and%
\begin{equation}
\int_{M}he^{2u_{1}}d\mu =\int_{M}he^{2u_{1}}d\mu =\rho ,
\end{equation}%
then%
\begin{equation}
\rho \geq 4\pi \left( 1-\Theta \right) .
\end{equation}
\end{example}

\end{document}